\documentclass[11pt]{amsart}

\usepackage{anysize} \marginsize{1.3in}{1.3in}{1in}{1in}
\let\OldMarginpar\marginpar
\renewcommand{\marginpar}[1]{\OldMarginpar{\footnotesize#1}}

\usepackage{floatrow}
\usepackage{xcolor}
\usepackage{amsmath}
\usepackage{mathtools}
\usepackage{enumerate}
\usepackage[all]{xy}
\usepackage[utf8]{inputenc}
\usepackage{varioref}
\usepackage{amsfonts}
\usepackage{amssymb}
\usepackage{bbm}
\usepackage{tikz-cd}
\usepackage{graphicx}
\usepackage{mathrsfs}
\usepackage[hypertexnames=false,pdftex,
pdfpagemode=UseNone,
breaklinks=true,
extension=pdf,
colorlinks=true,
linkcolor=blue,
citecolor=red,
urlcolor=blue,
]{hyperref}
\usepackage{cleveref}

\makeatletter
\@namedef{subjclassname@2020}{%
	\textup{2020} Mathematics Subject Classification}
\makeatother

\keywords{Differential ideal, Gr\"obner basis, multivariate jet, perfect graph, regular triangulation, polytope, Ehrhart polynomial.\\{This version of the manuscript was accepted for publication in the {\em Special Issue of the Journal of Symbolic Computation on the occasion of MEGA 2021}}}

\title[Combinatorial Differential Algebra of $x^p$]{Combinatorial Differential Algebra of $x^p$}

\author{Rida Ait El Manssour}
\address{Rida Ait El Manssour (Max-Planck-Institut f\"ur Mathematik in den Naturwissenschaften, Inselstra{\ss}e 22, 04103 Leipzig, Germany)} 
\email{rida.manssour@mis.mpg.de} 

\author{Anna-Laura Sattelberger}
\address{Anna-Laura Sattelberger (Max-Planck-Institut f\"ur Mathematik in den Naturwissenschaften, Inselstra{\ss}e 22, 04103 Leipzig, Germany)}
\email{anna-laura.sattelberger@mis.mpg.de}

\let\origmaketitle\maketitle
\def\maketitle{
	\begingroup
	\def\uppercasenonmath##1{} 
	\let\MakeUppercase\relax 
	\origmaketitle
	\endgroup
}

\newtheoremstyle{citing}
{}
{}
{\itshape}
{}
{\bfseries}
{\textbf{.}}
{.5em}
{\thmnote{#3}}

\usepackage{amsthm}

\theoremstyle{theorem}
\newtheorem{theorem}{Theorem}[section]
\newtheorem{question}[theorem]{Question}
\newtheorem{lemma}[theorem]{Lemma}
\newtheorem{corollary}[theorem]{Corollary}
\newtheorem{proposition}[theorem]{Proposition}

\theoremstyle{definition}

\newtheorem{definition}[theorem]{Definition}
\newtheorem{examplex}[theorem]{Example}
\newenvironment{example}
{\pushQED{\qed}\examplex}
{\popQED\endexamplex}

\newtheorem{remarkx}[theorem]{Remark}
\newenvironment{remark}
{\pushQED{\qed}\remarkx}
{\popQED\endremarkx}

\newtheorem*{claim}{Claim}

\newtheoremstyle{citing}
{}
{}
{\itshape}
{}
{\bfseries}
{\textbf{.}}
{.5em}
{\thmnote{#3}}
{\theoremstyle{citing}
	\newtheorem*{custom}{}}

\DeclareMathOperator{\NN}{\mathbb{N}}
\DeclareMathOperator{\CC}{\mathbb{C}}
\DeclareMathOperator{\RR}{\mathbb{R}}
\DeclareMathOperator{\ZZ}{\mathbb{Z}}
\DeclareMathOperator{\QQ}{\mathbb{Q}}
\DeclareMathOperator{\init}{in}
\DeclareMathOperator{\conv}{conv}
\DeclareMathOperator{\lm}{lm}
\DeclareMathOperator{\Spec}{Spec}
\DeclareMathOperator{\lcm}{LCM}
\DeclareMathOperator{\STAB}{Stab}
\DeclareMathOperator{\QSTAB}{QStab}

\begin{document}
\thispagestyle{empty}

\begin{abstract}
We link $n$-jets of the affine monomial scheme defined by $x^p$ to the stable set polytope of some perfect graph. We prove that, as $p$ varies, the dimension of the coordinate ring of a certain subscheme of the scheme of $n$-jets as a $\CC$-vector space is a polynomial of degree $n+1,$ namely the Ehrhart polynomial of the stable set polytope of that graph. One main ingredient for our proof is a result of Zobnin who determined a differential Gr\"obner basis of the differential ideal generated by~$x^p.$ We generalize Zobnin's result to the bivariate case. We study $(m,n)$-jets, a higher-dimensional analog of jets, and relate them to regular unimodular triangulations.
\end{abstract}

\maketitle
\setlength{\parindent}{1em}
\setcounter{tocdepth}{1}
\vspace*{-5mm}
\tableofcontents
\vspace*{-5mm}

\section*{Introduction}
\thispagestyle{empty} 
Differential algebra---an infinite version of polynomial algebra in a sense---studies polynomial partial differential equations with tools from commutative algebra. Differential algebraic geometry studies varieties that are defined by a system of polynomial PDEs. An upper bound for the number of components of such a variety was recently constructed in~\cite{upperboundDA}.  Differential algebraic geometry comes with an own version of the Nullstellensatz, the {\em differential} Nullstellensatz, relating points of a differential variety with formal power series solutions of the defining system of equations. Lower and upper bounds for the effective differential Nullstellensatz are provided in~\cite{diffNST}.
In this article, we transfer the combinatorial flavor of commutative algebra~\cite{combalg} to differential algebra and undertake first steps in combinatorial differential algebra. We present a case study of the fat point $x^p$ on the affine line. 

Denote by $C_{p,n}$ the ideal in $R_n=\CC[x_0,\ldots,x_n]$ generated by the coefficients of 
$f_{p,n}=(x_0+x_1t+\cdots+x_nt^n)^p,$ read as a polynomial in~$t$ with coefficients in $R_n.$ The affine scheme defined by $C_{p,n}$ is a subscheme of the scheme of $n$-jets of the fat point~$x^p$ on the affine line. 
Bernd Sturmfels suggested investigating the following question on the dimension of the coordinate ring of that scheme.
\begin{custom}[Question~\ref{conjBernd}]
	For fixed $n\in \NN,$  is the sequence $\left( \dim_{\CC} (R_n/C_{p,n}) \right)_{p\in \NN}$ a polynomial in~$p$ of degree $n+1$?
	\end{custom} The point of departure is experimental observations. A first main result of this article is~\Cref{thm Ehrhart}, proving that the answer to this question is {\em yes}. The dimension of this quotient ring provides a dimensional bound in the study of the multiplicity structure of the arc space of the fat point which was undertaken by the first-named author and Pogudin in~\cite{AEMP21}.

One main tool for the proof is a result from differential algebra. The object of study is the {\em ring of differential polynomials in $x$ over $\CC.$} The latter is the differential ring $\CC[x^{(\infty)}]=(\CC[x,x^{(1)},x^{(2)},\ldots],\partial),$ i.e., the polynomial ring in the countably infinitely many variables $\{x^{(k)}\}_{k\in \NN}$ with the derivation $\partial$ acting as $\partial(x^{(k)})=x^{(k+1)},$ $\partial|_{\CC}\equiv 0.$ An ideal $I$ in $\CC[x^{(\infty)}]$ is a {\em differential ideal} if $\partial(I)\subseteq I.$ Zobnin~\cite{Zobnin} proved that the singleton $\{x^p\}$ is a differential Gr\"obner basis of the differential ideal generated by $x^p$ with respect to certain monomial orderings, namely so called $\beta$-orderings. 
\\Denote by $I_{p,n}$ the differential ideal generated by $x^p$ and $x^{(n)}.$ Then the map
$$ R_n/C_{p,n}  \stackrel{\cong}{\longrightarrow} \CC[x^{(\infty)}]/I_{p,n+1}, \quad x_k \mapsto \frac{1}{k!}x^{(k)}$$
is an isomorphism and Zobnin's result can be used to investigate $C_{p,n}.$	
An investigation of the initial ideal of $C_{p,n}$ then reveals the following.
\begin{custom}[\Cref{thm Ehrhart}]
As $p$ varies, the number $\dim_{\CC} (R_n/C_{p,n})$ is polynomial of degree $n+1.$
It is the Ehrhart polynomial of the convex polytope $$P_n  \coloneqq  \left\{(u_0,\ldots,u_n)\in (\mathbb{R}_{\geq 0})^{n+1} \big| u_i+u_{i+1}\leq 1 \text{ for all } 0\leq i\leq n-1\right\}$$ 
evaluated at $p-1,$ i.e., it counts the lattice points of the polytope $P_n$ dilated by $p-1.$
\end{custom} 
In \cite{BMS}, the authors study the arc scheme of a variety, which is defined to be the projective limit of the jet schemes.  Arc schemes are an important tool from algebraic geometry to study formal power series solutions to the defining equations of a variety. Via Hilbert--Poincar\'{e} series, they bring a combinatorial flavor in, which is different from the one we present here; among others, they relate the double point to partitions of integers.
A study of jet schemes of monomial ideals was also undertaken in~\cite{GS06}. Therein, it is shown that jet schemes of monomial ideals are in general not monomial, but their reduced subschemes are. A study of the multiplicity of jet schemes of simple normal crossing divisors was undertaken by Yuen in~\cite{YuenSNC}. In~\cite{YuenWedge}, she introduced {\em truncated $m$-wedges}, a two-dimensional analog of jets, studying differentials in two variables whose orders add up to~$m$ at most.  In Definition~\ref{def mnjets}, we introduce another generalization of jets to higher dimensions, namely {\em $(m,n)$-jets}, allowing for derivatives in the variables up to order $m$ and $n,$ respectively. 

 We extend our studies of $\dim_{\CC}(R_n/C_{p,n})$ to the case of two independent variables and give a link to regular unimodular triangulations. For the theory of triangulations, we refer our readers to~\cite{triang, BerndGrobner}. 
We study the ring of {\em partial} differential polynomials in $x$ over $\CC$, i.e.,
 $\CC[x^{(\infty,\infty)}]  \coloneqq  \left( \CC[x^{(k,\ell) }]_{k,\ell \in \NN}, \partial_s,\partial_t\right)$ in two independent variables $s,t$ and  consider the differential ideal $I_{p,(m,n)}$ generated by $x^p, x^{(m,0)},$ and $x^{(0,n)}.$
 Denote by $C_{p,(m,n)}$ the ideal in $\CC[\{x_{k,\ell}\}_{0\leq k \leq m,0\leq \ell\leq n}]$ generated by
  the coefficients of  
 $$ f_{p,(m,n)} \coloneqq \left( \sum_{k=0}^{m}\sum_{\ell=0}^n x_{k,\ell}t^ks^{\ell} \right)^p ,$$
 read as a bivariate polynomial in $s$ and $t.$ In analogy to the one-dimensional case, we refer to points of the affine scheme associated to $C_{p,(m,n)}$ as {\em $(m,n)$-jets} of $x^p.$ The ideals $I_{p,(m,n)}$ and $C_{p,(m,n)}$ then are related just as in the univariate case. 
  
 For a triangulation $T$ of the $m\times n$ rectangle and fixed $p,$ we define {\em $T$-orderings} on the truncated partial differential ring $\CC[x^{(\leq m,\leq n)}]$ as those monomial orderings for which  the leading monomials of $\{(x^p)^{(k,\ell)}\}_{k=0,\ldots,mp,\ell = 0,\ldots,np}$ are supported on the triangles of~$T.$ Note that this is in contrast to the usual occurrence of regular triangulations in combinatorial commutative algebra, where the leading monomials are supported on {\em non-}faces (see  Sturmfels' correspondence \cite[Theorem 9.4.5]{triang}). We~consider the placing triangulation $T_{m,2}$  of the point configuration $$[(0,0),(0,1),(0,2),(1,0),(1,1),(1,2),\ldots,(m,0),(m,1),(m,2)].$$ This is a regular  unimodular triangulation of the $m\times 2$ rectangle induced by the vector $(1,2,2^2,\ldots, 2^{3m+2})$ in the lower hull convention. 
 
We prove the following generalization of Zobnin's result to the partial differential ring in two independent variables.
 \begin{custom}[\Cref{thm zobnin m2}]
 	For  all $m, p\in \NN,$ $\{(x^p)^{(k,\ell)}\}_{0\le k\le mp, 0\le \ell \le 2p}$ is a Gr\"obner basis of the differential ideal generated by  $ x^p $  in the truncated partial differential ring $\CC[x^{(\le m,\le 2)}]$ with respect to any $T_{m,2}$-ordering. 
 \end{custom}
This theorem is the main ingredient for the following theorem.
	\begin{custom}[\Cref{prop dimmultivar}]
	Let $m,p \in \NN.$ The number $\dim_{\CC} ( R_{m,2}/C_{p,(m,2)})$ is the Ehrhart polynomial of the \mbox{$3(m+1)$-dimensional} lattice polytope
	\begin{align*}
	P_{(m,2)}  \coloneqq  &\big\{(u_{00},u_{01},u_{02},\ldots,u_{m0},u_{m1},u_{m2})\in (\mathbb{R}_{\geq 0})^{3(m+1)} \big| u_{k_1,l_1}+u_{k_2,l_2}+u_{k_3,l_3}\leq 1\\ &\ \text{for all indices s.t. }\{(k_1,l_1),(k_2,l_2),(k_3,l_3)\} \text{ is a triangle of } T_{m,2} \big\}
	\end{align*}
	evaluated at~$p-1.$
\end{custom}
In \Cref{section regtrg}, we study regular unimodular triangulations of the $m \times n$ rectangle.  We consider the weighted reverse lexicographical ordering on $\CC[\{x_{k,\ell}\}_{0\leq k \leq m,0\leq \ell\leq n}]$ for vectors inducing those triangulations in the upper hull convention. We show that for some of them, the coefficients of $f_{p,(m,n)}$ are a Gr\"obner basis of the ideal~$C_{p,(m,n)}.$

We end our article with an outlook to future work. Our results suggest to further develop combinatorial differential algebra. 

\medskip
\noindent \textbf{Acknowledgments.} We thank Bernd Sturmfels for suggesting the problem and for insightful discussions. We are grateful to Alex Fink and Gleb Pogudin for useful discussions and to Taylor Brysiewicz, Lars Kastner, Marta Panizzut, and Paul Vater for their great support with computations in {\tt polymake}~\cite{polymake} and {\tt TOPCOM}~\cite{topcom}. Last but not least, we  thank  Nero Budur for pointing out the related article~\cite{af20} to us. 

\section{Differential ideals and jets}\label{section diffideals}
\subsection{One independent variable} In this section, we repeat basics from differential algebra and give a link to the theory of jet schemes. For further background on differential algebra, we refer the reader to the books~\cite{Kolchin,Ritt}.

Consider the polynomial ring $\CC[x,x^{(1)},x^{(2)},\ldots]$ in the countably infinitely many variables $\{x^{(k)}\}_{k\in \NN}.$ Denoting $x\coloneqq x^{(0)},$ let $\CC[x^{(\infty)}]$ be the differential ring $$ \CC[x^{(\infty)}]  \coloneqq  \left(\CC[x,x^{(1)},x^{(2)},\ldots],\partial\right),$$
where the derivation acts as $\partial(x^{(k)})=x^{(k+1)}$ and $\partial|_{\CC} \equiv 0.$  
\begin{definition}
An ideal $I\triangleleft \CC[x^{(\infty)}]$ is called a {\em differential ideal} if $\partial(I)\subseteq I.$ For a subset $J$ of $\CC[x^{(\infty)}],$ we denote by $\langle J \rangle ^{(\infty)}$ the differential ideal generated by $J.$
\end{definition}
\noindent We denote by $I_{p,n}\coloneqq \langle x^p,x^{(n)}\rangle^{(\infty)}$ the differential ideal  in $\CC[x^{(\infty)}]$ generated by $x^p$ and $x^{(n)}$ and by $\CC[x^{(\leq n)}]$ the {\em truncated} differential polynomial ring $\CC[x^{(\infty)}]/\langle x^{(n+1)}\rangle^{(\infty)}.$

For $n\in \NN,$ denote by $$ R_n  \coloneqq  \CC[x_0,\ldots,x_n]$$ the polynomial ring in $n+1$ variables with coefficients in the complex numbers. 
Consider $f_{p,n}=(x_0+x_1t+\cdots+x_nt^n)^p \in R_n[t].$ By the multinomial theorem,
$$ f_{p,n} = \sum_{k_0+\cdots +k_n =p} \binom{p}{k_0,k_1,\ldots,k_n}x_0^{k_0}\cdots x_n^{k_n}t^{k_1+2k_2+\cdots + nk_n},$$
where $$ \binom{p}{k_0,k_1,\ldots,k_n}  = \frac{p!}{k_0!\cdots k_n!}.$$

Denote by $C_{p,n} \triangleleft R_n$ the ideal generated by the coefficients of $f_{p,n}.$
The affine scheme $\Spec (R_n/ C_{p,n} )$ is a subscheme of the scheme of $n$-jets of~$\Spec (\CC[x]/\langle x^p\rangle ).$ 

Up to constants, the coefficient of $t^k$ in $f_{p,n}$ recovers the $k$-th derivative of the monomial~$x^p,$ giving rise to the following relation between the differential ideal $I_{p,n}$ and the ideal~$C_{p,n}$ in the polynomial ring $R_n.$
\begin{proposition}\label{prop iso}
The following map is an isomorphism of $\CC$-algebras:
$$ R_n/C_{p,n}  \stackrel{\cong}{\longrightarrow} \CC[x^{(\infty)}]/I_{p,n+1}, \quad x_k \mapsto \frac{1}{k!}x^{(k)}.$$
\end{proposition}
\begin{proof}
Notice that $(x^p)^{(k)}$ is given  as follows:
$$(x^p)^{(k)} = \sum_{j_0+\cdots +j_{p-1}=k} \binom{k}{j_0,\ldots,j_{p-1} } x^{(j_0)}\cdots x^{(j_{p-1})}.$$
Let us consider its image in the truncated differential ring $\CC[x^{(\leq n)}].$ We denote by $i_{\ell}$ the multiplicity of~$\ell$ in the multiset $\left\{j_0,\ldots ,j_{p-1}\right\},$ so that $i_0+\cdots +i_n=p$ and $i_1+2i_2+\cdots +ni_n=k.$ Let $y_i:=x^{(i)}$ for all $0\le i\le n.$ Then
$$(x^{p})^{(k)} = \sum_{j_0+\cdots +j_{p-1}=k} \binom{k}{j_0,\ldots, j_{p-1}} y_0^{i_0}\cdots y_n^{i_n}.$$
In the previous sum, there are some repeated terms: for each $\{j_0,\ldots, j_{p-1}\}$ by exchanging the order of $j_i$ and respecting the numbers $i_0,\ldots, i_n,$ we get the same term. We have  $\binom{p}{i_0}$ possibilities to choose $i_0$ many places for $0$ in the multiset $\{j_0,\ldots j_{p-1}\}.$ We have  $\binom{p-i_0}{i_1}$ possibilities to choose $i_1$ many places for $1$ from the remaining places in the set $\{j_0,\ldots, j_{p-1}\}.$ We continue like this and obtain
$$(x^p)^{(k)} = \sum_{(i_0,\ldots,i_n)\in I} \binom{p}{i_0,\ldots, i_n} \cdot \frac{k!}{(0!)^{i_0}\cdots (n!)^{i_n}}\cdot y_0^{i_0}\cdots y_n^{i_n},$$ where $I=\{(i_0,i_1,\ldots,i_p)\mid i_0+\cdots+ i_n=p \text{ and }i_1+\cdots +ni_n=k\}.$
Denote by $\varphi$ the following homomorphism of rings:
$$\varphi  \colon  \CC[x^{(\leq n)}] \to R_n/C_{p,n}, \quad  x^{(k)} \mapsto  k!\cdot x_k.$$
The kernel of $\varphi$ is the ideal generated by $\{(x^p)^{(k)}\}_{k\in \NN}.$ Thus,
$$ \CC[x^{(\infty)}]/I_{p,n+1} \cong \CC[x^{(\leq n)}]/\langle \{(x^p)^{(k)}\mid k\in \NN\}\rangle\cong R_n/C_{p,n},$$
concluding the proof.
\end{proof}
\begin{remark} 
The statement of \Cref{prop iso} is contained in the literature, such as \cite[Proposition 2.3]{Mourtada} or \cite[Proposition 5.12]{MS10}. To make this article self-contained, we decided to provide a proof.
\end{remark}
Following~\cite{Zobnin}, we now repeat the concept of {\em differential Gr\"{o}bner bases}.

\begin{definition}\label{def diff standard basis} 
	Fix a monomial ordering $\prec$ on $\CC[x^{(\infty)}]$ and let $I\triangleleft\CC[x^{(\infty)}]$ be a differential ideal. A subset of polynomials  $G\subseteq I$ s.t. $\langle G\rangle^{(\infty)}=I$ is a {\em differential Gr\"{o}bner basis of $I$} if $\{\partial^k (g)\mid k\in \NN,g\in G\}$ is an algebraic Gr\"{o}bner basis of \mbox{$I\triangleleft \CC[x,x^{(1)},x^{(2)},\ldots]$} with respect to $\prec.$
\end{definition} 
\begin{remark}
	We thank the anonymous referee for pointing us to the related articles~\cite{HS09,NR,KLS} and  Arthur Bik and Aida Maraj for a helpful discussion of those. The setup of $\text{Inc}(\NN)$-stable ideals has parallels to the theory we study. Yet, the results cannot be immediately transferred from one to the other; in differential algebra, one has, for instance, fewer Noetherianity properties at one's disposal and Leibniz' rule instead of multiplicativity, which leads to a structurally different behavior. It would be intriguing to investigate in which extent results from one area transfer to the other.
\end{remark}

Zobnin studied the differential ideal $\langle x^p\rangle^{(\infty)}$ and proved the following.
\begin{theorem}[{\cite{Zobnin}}]\label{thm Zobnin}
The singleton $\{x^p\}$ is a differential Gr\"{o}bner basis of $\langle x^p\rangle^{(\infty)}$ with respect to the reverse lexicographical ordering.
\end{theorem}
\begin{remark}
Zobnin proved this result for so called  {\em $\beta$-orderings}, i.e., monomial orderings on $\CC[x^{(\infty)}]$ for which the leading monomial of $(x^p)^{(k)}$ is of the form \linebreak$(x^{(i)})^a(x^{(i+1)})^{p-a}$ (see~\cite{Levi}). Since, in this article, we do not need the statement in its full generality, we just point out that the reverse lexicographical ordering is such a \mbox{$\beta$-ordering}. Note moreover that $(x^p)^{(k)}$ is bihomogeneous with respect to the vectors $(1,1,1,\ldots)$ and $(0,1,2,3,\ldots),$ i.e., every monomial summand $\prod_i (x^{(i)})^{a_i}$ in $(x^p)^{(k)}$ satisfies $\sum_i a_i=p$ and $\sum_i ia_i=k.$ 
\end{remark}

\subsection{Two independent variables} 
We now generalize \Cref{prop iso} to two independent variables. We denote by $\CC[x^{(\infty,\infty)}] $ the ring of partial differential polynomials in $x$ over $\CC$, i.e., the differential ring
 $$\CC[x^{(\infty,\infty)}]  \coloneqq  \left( \CC[x^{(k,\ell) }]_{k,\ell \in \NN}, \partial_s,\partial_t\right)$$
in the two independent variables $s,\,t$ and the commuting derivations $\partial_s,\,\partial_t$ acting as
$$ \partial_s(x^{{(k,\ell)}}) =x^{(k+1,\ell)}, \quad \partial_t(x^{(k,\ell)}) =x^{(k,\ell+1)}, \quad \partial_s|_{\CC}\equiv \partial_t|_{\CC}\equiv 0.$$
For $m,n\in \NN,$ denote by $I_{p,(m,n)}$ the differential ideal $\langle x^p,x^{(m,0)},x^{(0,n)}\rangle^{(\infty,\infty)}$ in \linebreak $\CC[x^{(\infty,\infty)}]$ generated by $x^p,x^{(m,0)},$ and $x^{(0,n)}.$

Denote by $R_{m, n}$ the polynomial ring in the $(m+1)(n+1)$ many variables \linebreak$\{x_{k,\ell}\}_{0\leq k \leq m,0\leq \ell \leq n}$ and let  $f_{p,(m,n)}$ be the bivariate polynomial $$f_{p,(m,n)} \coloneqq \left( \sum_{k=0}^{m}\sum_{\ell=0}^n x_{k,\ell}t^ks^{\ell} \right)^p \in R_{m,n}[s,t].$$ 
By the multinomial theorem,
$$ f_{p,(m,n)} = \sum_{\sum i_{k,\ell} = p }\left[ \binom{p}{i_{0,0},\ldots,i_{m,n}} \cdot \prod_{k,\ell}x_{{k,\ell}}^{i_{k,\ell}}s^{k\cdot i_{k,\ell}}\cdot t^{\ell \cdot i_{k,\ell}}\right],$$
where $(k,\ell)\in \{0,\ldots,m\}\times \{0,\ldots n\}$ and $i_{k,\ell}\in \NN$ for all $(k,\ell).$
Let $C_{p,(m,n)}\triangleleft R_{m,n}$ denote the ideal generated by the  the coefficients of $f_{p,(m,n)}.$

\begin{definition}\label{def mnjets}
We refer to points of $\Spec (R_{m,n}/C_{p,(m,n)})$ as {\em $(m,n)$-jets} of the affine monomial scheme defined by $x^p.$
\end{definition}
\begin{proposition}\label{prop iso partial}
	The following map is an isomorphism of $\CC$-algebras:
	$$ R_{m,n}/C_{p,(m,n)}  \stackrel{\cong}{\longrightarrow}\CC[x^{(\infty,\infty)}] /I_{p,(m+1,n+1)}, \quad x_{k,\ell}\mapsto  \frac{1}{k!\ell !}\cdot x^{(k,\ell)}.$$
\end{proposition}
\begin{proof}
The coefficient $f_{a,b}$ of $s^at^b$ in $f_{p,(m,n)}$ is given as 
$$ f_{a,b} = \sum_{ (i_{k,\ell} ) \in I } \left[ \binom{p}{i_{0,0},\ldots,i_{m,n}} \cdot \prod_{k,\ell} x_{k,\ell}^{i_{k,\ell}}\right],$$
	where $I = \{ (i_{k,\ell})_{k,\ell} \mid  \sum_{k=0}^m (k\sum_{\ell=0}^n i_{k,\ell})=a,\,\sum_{\ell=0}^n(\ell\sum_{k=0}^m i_{k,\ell}) =b,\, \sum i_{k,\ell}=p\}.$\linebreak
	By the symmetry of the second derivatives, we obtain
\small
\begin{align*} &(x^p)^{(a,b)} = \left((x^p)^{(a,0)}\right)^{(0,b)} \\
&\quad = \left( \sum_{\sum_{i=0}^m k_i=p, k_1+2k_2+\cdots mk_m=a}\binom{p}{k_0,\ldots,k_m} \frac{a!}{(0!)^{k_0}\cdots (m!)^{k_m}} \cdot  (x^{(0,0)})^{k_0} \cdots (x^{(m,0)})^{k_m} \right)^{(0,b)} \\
 &\quad  = \sum_{k_0,\ldots,k_m}\binom{p}{k_0,\ldots,k_m}\frac{a!}{(0!)^{k_0}\cdots(m!)^{k_m}} \sum_{\ell_0+\cdots + \ell_{p-1}=b } \binom{b}{\ell_0,\ldots,\ell_{p-1}} \\ &\quad\qquad 
 \prod_{0\le i\le m } x^{(i,\ell_{k_0+\ldots+k_{i-1}})}\cdots x^{(i,\ell_{k_0+\ldots+k_{i}-1})}.
\end{align*}
\normalsize
For all $0\le i\le m$ and $0\le s\le n,$ let $j_i^s$ be the multiplicity of $s$ in the multiset $\{\ell_{k_0+\cdots +k_{i-1}},\ldots, \ell_{k_0+\cdots +k_{i}-1}\}.$ Thus $k_i=\sum_{s=0}^nj_i^s,$ $\sum_{i,s}j_i^s=p,$  $\sum_{i,s}ij_i^s=a,$ and $\sum_{i,s}sj_i^s=b.$ Let $J$ denote the set of all those $(j_i^s)_{s,i}.$ Then$(x^p)^{(a,b)}$ equals
\small
\begin{align*}
&\sum_{(j_i^s)\in J} \left[ \frac{p!}{k_0!\cdots k_m!}\frac{a!}{(0!)^{k_0}\cdots (m!)^{k_m}}\frac{b!}{(0!)^{\sum j_i^0}\cdots (n!)^{\sum j_i^n}}\frac{k_0!}{j_0^0!\cdots j_0^n!}\cdots \frac{k_m!}{j_m^0!\cdots j_m^n!}\prod_{i,s}\left(x^{(i,s)}\right)^{j_i^s}\right]\\
	&\quad =\sum_{(j_i^s)\in J} \left[\binom{p}{i_{0,0}\cdots i_{m,n}}a!b!\prod_{i,s}\frac{\left(x^{(i,s)}\right)^{j_i^s}}{i!s!}\right],
\end{align*}\normalsize
concluding the proof.
\end{proof}

In order to generalize \Cref{thm Zobnin} to {\em partial} differential rings, we first generalize the concept of $\beta$-orderings. We denote by $\CC[x^{(\leq m,\leq n)}]$ the truncated differential ring $\CC[x^{(\infty,\infty)}]/\langle x^{(m,0)},x^{(0,n)}\rangle^{(\infty,\infty)}.$ A monomial is {\em supported} on a triangle of $T$ if the indices of the variables of that monomial are among the vertices of a triangle of $T.$
\begin{definition}
Fix $m,n,p\in \NN$ and a triangulation $T$ of the $m\times n$ rectangle. A monomial ordering $\prec$ on $\CC[x^{(\leq m,\leq n)}]$ is a \mbox{\em $T$-ordering} if the leading monomial of each $(x^p)^{(k,\ell)},$ $0\le k\le mp,$ $0\le \ell \le np,$ is supported on a triangle of~$T.$
\end{definition}

\begin{remark}
By identifying $x^{(k,\ell)}$ with $k!\ell!x_{k,\ell},$ one equivalently defines a \mbox{\em $T$-ordering}  as a monomial ordering on  $R_{m,n}=\CC [\{x_{k,\ell}\}_{0\leq k\leq m,0\leq \ell \leq n}]$  such that the leading monomial of each coefficient of $f_{p,(m,n)}\in R_{m,n}[s,t]$ is supported on a triangle of~$T.$
\end{remark}
Denote by $T_{m,2}$ the unimodular triangulation of the $m\times 2$ rectangle depicted in \Cref{fig reg triang m2}.
\begin{figure}
	$\begin{tikzpicture}[scale=1.2]
	\filldraw (0,0) circle (2pt) node[below, scale=0.8] {$(0,0)$};
	\filldraw (1,0) circle (2pt) node[below, scale=0.8] {$(1,0)$};
	\filldraw (2,0) circle (2pt) node[below, scale=0.8] {$(2,0)$};
	\filldraw (0,1) circle (2pt) node[left, scale=0.8] {$(0,1)$};
	\filldraw (0,2) circle (2pt) node[above, scale=0.8] {$(0,2)$};
	\filldraw (1,1) circle (2pt);
	\filldraw (1,2) circle (2pt) node[above, scale=0.8] {$(1,2)$};
	\filldraw (2,1) circle (2pt);
	\filldraw (2,2) circle (2pt) node[above, scale=0.8] {$(2,2)$};
	\filldraw (3,1) circle (2pt);
	\filldraw (3,2) circle (2pt) node[above, scale=0.8] {$(3,2)$};
	\filldraw (4,1) circle (2pt);
	\filldraw (4,2) circle (2pt);		
	\draw(4,2.05)  node[above, scale=0.8] {$\cdots$};
	\filldraw (5,1) circle (2pt) node[right, scale=0.8] {$(m,1)$};
	\filldraw (5,2) circle (2pt) node[above, scale=0.8] {$(m,2)$};
	\filldraw (3,0) circle (2pt) node[below, scale=0.8] {$(3,0)$};
	\filldraw (4,0) circle (2pt);
	\draw (4,-0.1)  node[below, scale=0.8] {$\cdots$};
	\filldraw (5,0) circle (2pt) node[below, scale=0.8] {$(m,0)$};
	\draw (0,0)--(5,0);
	\draw (0,0)--(0,2);
	\draw  (1,0)--(0,1);
	\draw  (0,2)--(0,1);
	\draw  (2,0)--(2,2);
	\draw  (0,2)--(5,2);
	\draw (0,2)--(2,0);
	\draw (1,0)--(1,2);
	\draw  (1,0)--(0,2);
	\draw  (2,1)--(1,2);
	\draw (1,2)--(2,0);
	\draw  (3,0)--(3,2);
	\draw  (4,0)--(4,2);
	\draw  (5,0)--(5,2);
	\draw (2,2)--(3,0);
	\draw (3,2)--(4,0);
	\draw (4,2)--(5,1);
	\draw (2,2)--(3,1);
	\draw (4,2)--(5,0);
	\draw (3,2)--(4,1);
	\draw (2,1)--(3,0);
	\draw (3,1)--(4,0);
	\draw (4,1)--(5,0);
	\end{tikzpicture}$
	\caption{\small The placing triangulation $T_{m,2}$ of the $m \times 2$ rectangle of the point configuration $[(0,0),(0,1),(0,2),(1,0),(1,1),(1,2), \ldots, (m,0),(m,1),(m,2)].$\normalsize}
	\label{fig reg triang m2}
\end{figure}
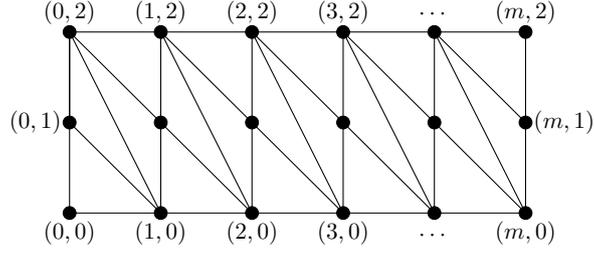 
This is the placing triangulation of the point configuration $$[(0,0),(0,1),(0,2),(1,0),(1,1),(1,2),\ldots,(m,0),(m,1),(m,2)].$$
Note that the vector $(1,2,2^2,\ldots, 2^{3m+2})$ induces the triangulation $T_{m,2}$ in the {\em lower} hull convention, hence $T_{m,2}$ is a regular triangulation. Denote by $T_{m,n}$ the placing triangulation of $[(0,0),\ldots,(0,n),(1,0),\ldots,(1,n),\ldots,(m,0),\ldots,(m,n)].$ It consists of $m$ copies of the triangulation in \Cref{fig reg1n}.
	\begin{figure}[h]
	\floatbox[{\capbeside\thisfloatsetup{capbesideposition={right, center},capbesidewidth=11cm}}]{figure}[\FBwidth]
	{\caption{\small The regular placing triangulation $T_{1,n}$ of the $1 \times n$ rectangle for the point configuration $[(0,0),(0,1),\ldots,(0,n),(1,0),\ldots,(1,n)].$\normalsize}	\label{fig reg1n}}{
		$\hspace*{1cm}\begin{tikzpicture}[scale=1.2]
		\filldraw (0,0) circle (2pt) node[left, scale=0.8] {$(0,0)$};
		\filldraw (1,0) circle (2pt) node[right, scale=0.8] {$(1,0)$};
		\filldraw (0,1) circle (2pt) node[left, scale=0.8] {$(0,1)$};
		\filldraw (0,2) circle (2pt) node[left, scale=0.8] {$(0,2)$};
		\filldraw (0,3) circle (2pt);
		\draw (-0.2,3)  node[left, scale=0.8] {\vdots};
		\filldraw (0,4) circle (2pt) node[left, scale=0.8] {$(0,n)$};
		\filldraw (1,4) circle (2pt) node[right, scale=0.8] {$(1,n)$};
		\filldraw (1,3) circle (2pt);	
		\draw (1.2,3)  node[right, scale=0.8] {\vdots};
		\filldraw (1,1) circle (2pt) node[right, scale=0.8] {$(1,1)$};
		\filldraw (1,2) circle (2pt) node[right, scale=0.8] {$(1,2)$};
		\draw (1,0)--(1,4);
		\draw (0,0)--(1,0);
		\draw (1,0)--(0,4);
		\draw (1,0)--(0,3);
		\draw (0,4)--(1,4);
		\draw (0,4)--(1,1);
		\draw (0,4)--(1,2);
		\draw (0,4)--(1,4);
		\draw (0,0)--(0,4);
		\draw  (1,0)--(0,1);
		\draw  (0,4)--(0,1);
		\draw  (0,3)--(0,1);
		\draw  (1,0)--(0,2);
		\draw  (0,4)--(1,0);
		\draw  (0,4)--(1,0);
		\draw  (0,4)--(1,3);
		\end{tikzpicture}\hspace*{-1cm}$}
	\end{figure}
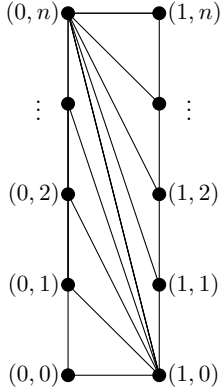
\begin{proposition}\label{prop Tm2ord}
For all $0\le k\le mp,$ and $0\le \ell\le np,$ $(x^p)^{(k,\ell)}$ has a unique monomial summand supported on a triangle of $T_{m,n}.$ Moreover, the reverse lexicographical ordering~$\prec$ on $\CC[x^{(0,0)},x^{(0,1)},\ldots,x^{(0,n)},\ldots,x^{(m,0)},\ldots,x^{(m,n)}]$ is a $T$-ordering for $T=T_{m,n}$ for all $p,$ where we order the variables as $x^{(0,0)}\prec x^{(0,1)}\prec \cdots\prec x^{(0,n)}\prec x^{(1,	0)}\prec \cdots \prec x^{(1,n)} \prec \cdots \prec x^{(m,0)} \prec \cdots \prec x^{(m,n)}.$
\end{proposition}
\begin{proof}
Consider $(x^p)^{(k,\ell)}$ and let us suppose that it has a monomial summand supported on a triangle of $T$  of the form $x_{h,n}^ax_{h+1,s}^bx_{h+1,s+1}^c.$ Suppose that there exists a monomial $M=\prod_j x_{j,0}^{i_{j,0}}\cdots x_{j,n}^{i_{j,n}}$ in $(x^p)^{(k,\ell)}$ such that $x_{h,n}^ax_{h+1,s}^bx_{h+1,s+1}^c \prec M.$ Then, it follows that  $i_{h,n} \le a,$ $i_{h,0}=\cdots =i_{h,n-1}=0,$  and  $i_{j,0}=\cdots=i_{j,n}=0$ for all $j<h.$ Moreover, the following identities hold:
\begin{align}\label{3equalities}
\begin{split}
a+b+c=\sum_{j\ge h} i_{j,0}+\cdots +i_{j,n}=p,\\
ha+(h+1)b+(h+1)c=\sum_{j\ge h}j(i_{j,0}+\cdots +i_{j,n})=k,\\
na+sb+(s+1)c=\sum_{j\ge h}i_{j,1}+\cdots +ni_{j,n}=\ell.
\end{split}
\end{align}
Then from the first two lines in~\eqref{3equalities}, we obtain
\small \[(a-i_{h,n})+(h+1)\big(\sum_{j\ge h}(i_{j,0}+\cdots+i_{j,n})-p\big)+\sum_{j\ge h+2}(j-h-1)(i_{j,0}+\cdots+i_{j,n})=0.\] \normalsize
Note that each of the three summands on the left hand side is nonnegative. Thus $$M=x_{h,n}^{i_{h,n}}x_{h+1,0}^{i_{h+1,0}}\cdots x_{h+1,n}^{i_{h+1,n}}, \ i_{h,n}=a, \text{ and } i_{h+1,0}+\ldots+i_{h+1,n}=b+c.$$ Since  $x_{h,n}^ax_{h+1,s}^bx_{h+1,s+1}^c \prec M,$  we have $i_{h+1,s}\le b,$ and for all $r<s, i_{h+1,r}=0.$ Then  from the third line in~\eqref{3equalities} we have $$s(b+c)+c=s(i_{h+1,s}+\cdots +i_{h+1,n})+i_{h+1,s+1}+\cdots+(n-s)i_{h+1,n}.$$ Thus $b=i_{h+1,s}$ and $c=i_{h+1,s+1}.$ We conclude that $M=x_{h,n}^ax_{h+1,s}^bx_{h+1,s+1}^c.$ 

Now suppose there exists a  monomial summand of $(x^p)^{(k,\ell)}$ that is supported on a triangle of $T_{m,n}$ of the form $x_{h,s}^ax_{h,s+1}^bx_{h+1,0}^c$ and suppose that  there exists a monomial $M$ such that $x_{h,s}^ax_{h,s+1}^bx_{h+1,0}^c \prec M.$ Then  $i_{h,s}\le a,$ $i_{h,r}=0$  for all $r<s,$ and
\begin{align}\label{3equations2}
\begin{split}
a+b+c=\sum_{j\ge h} i_{j,0}+\cdots+i_{j,n}=p,\\
ha+hb+(h+1)c=\sum_{j\ge h}j(i_{j,0}+\cdots+i_{j,n})=k,\\
sa+(s+1)b=\sum_{j\ge h}i_{j,1}+\cdots+ni_{j,n}=\ell.
\end{split}\end{align}
Suppose $a+b<i_{h,s}+\cdots+i_{h,n}.$ Then $b<i_{h,s+1}+\cdots+i_{h,n}.$ By the third line in~\eqref{3equations2},
\begin{equation}\label{3rdline}
s(a+b)+b=s(i_{h,s}+\cdots+i_{h,n})+(i_{h,s+1}+\cdots+(n-s)i_{h,n})+\sum_{j\ge h+1}i_{j,1}+\cdots+ni_{j,n},
\end{equation}
which is a contradiction to our assumption. 
From the first two lines in~\eqref{3equations2} we obtain
\footnotesize
\[(a+b-(i_{h,s}+\cdots+i_{h,n}))+(h+1)\left(\sum_{j\ge h}(i_{j,0}+\cdots+i_{j,n})-p\right)+\sum_{j\ge h+2}(j-h-1)(i_{j,0}+\cdots+i_{j,n})=0.\] \normalsize
Thus $a+b=i_{h,s}+\cdots+i_{h,n},$ and $c=i_{h+1,0}+\cdots+i_{h+1,n}.$
Therefore, from \eqref{3rdline} we conclude that $a=i_{h,s}, b=i_{h,s+1},$ and $c=i_{h+1,0}$ which means $M=x_{h,s}^ax_{h,s+1}^bx_{h+1,0}^c.$

We have proved that if $(x^p)^{(k,\ell)}$ contains a monomial summand supported on a triangle of~$T_{m,n},$ then this monomial is its leading monomial. Hence for all  $0\le k\le mp,$ $0\le \ell \le np,$ $(x^p)^{(k,\ell)}$ has at most one monomial summand that is supported on a triangle of $T_{m,n}.$
The triangles of $T_{m,n}$ are given by $\{(j,n),(j+1,s),(j+1,s+1)\}$ and $\{(j+1,0),(j,s),(j,s+1)\}$ for $0\le j\le m-1$ and $s=0, \ldots,n-1.$ The number of monomials of degree $p$ that are supported on these triangles is $(mp+1)(np+1).$ Indeed, we have $2nm$ triangles, $(3n+1)m+n$ edges, and $(n+1)(m+1)$ vertices on~$T_{m,n}.$ The number of monomials which can be formed by  the $2nm$ triangles containing all three corresponding variables  is $2nm\cdot\#\{a+b+c=p \, |\, a,b,c>0\}= 2nm\tfrac{(p-1)(p-2)}{2}.$ The edges give rise to $((3n+1)m+n)(p-1)$ monomials of degree $p$ in which both variables appear. The vertices give rise to $(n+1)(m+1)$ monomials of degree $p$ containing only this variable. Thus, $2nm\tfrac{(p-1)(p-2)}{2}+((3n+1)m+n)(p-1)+(n+1)(m+1)=(np+1)(mp+1)$ monomials are supported on triangles of $T_{m, n}.$ Each of these monomials belongs to the set of monomials appearing in the expression of $(x^p)^{(k,\ell)} $ for some $0\le k\le mp$ and $0\le \ell \le np.$  We conclude that every $(x^p)^{(k,\ell)}$ has exactly one monomial that is supported on a triangle of $T_{m,n}$ and this monomial is its leading monomial.
\end{proof}
We generalize Zobnin's result to the case of two independent variables as follows.
\begin{theorem}\label{thm zobnin m2}
For all $m,p\in \NN ,$  $\{(x^p)^{(k,\ell)}\}_{0\le k\le mp, 0\le \ell \le 2p}$ is a Gr\"obner basis of the differential ideal generated by  $ x^p $  in the truncated partial differential ring $\CC[x^{(\le m,\le 2)}]$ with respect to any $T_{m,2}$-ordering.
	\end{theorem}
\begin{proof}
	We pick up and generalize the construction from the proof of~\cite[Lemma~$1$]{Pog14}. Consider a $\CC$-vector space $V_p$ with $p-1$ pairs of countable series of basis vectors $\xi_{i,s}^r$ and  $\eta_{i,s}^r,$ where $r=0,\ldots,p-2,$ and $i,s\in \NN.$  Let $\Lambda(V_p)$ denote the exterior algebra of~$V_p.$ The two derivations increase the first (resp. second) index of the variables $\xi^r_{i,s}$ and $\eta^r_{i,s}$ by one. Extending this rule by Leibniz' rule endows $\Lambda(V_p)$ with the structure of a differential algebra. Denote by $M$ the set of monomials that are not divisible by any monomial of the form $\lm_\prec((x^p)^{(k,\ell)}).$ We will  prove that $M$ is a basis of $\CC[x^{(\infty,\infty)}]/I_{p,(m+1,3)}.$ Let $\phi:  \mathbb{C}[x^{(\le m,\le 2)}] \to \Lambda(V_p) $ be the differential homomorphism such that $\phi(x)=\sum_r \eta_{0,0}^r\wedge \xi_{0,0}^r$ and let $L=\sum_{K\in M} c_K K$ be a $\CC$-linear combination of elements of~$M.$ Notice that $\phi(x^p)=0.$ Hence we only need to prove that $\phi(L)\neq 0.$  Let $K=x_{i_1,s_1}x_{i_2,s_2}\cdots x_{i_k,s_k}$ be the leading monomial of $L$ w.r.t. the degrevlex~ordering.  
	\begin{claim}
	The variables $x_{i_1,s_1},x_{i_2,s_2},\ldots, x_{i_k,s_k}$   appearing  in the monomial $K$ can be distributed to sets $S_0,\ldots,S_\nu$ with $\nu \leq p-1$ such that each $S_r$ contains neither a repeated variable nor two variables that are supported on an edge  of $T_{m,2}.$ 
		\end{claim} 
	\begin{proof}
		We prove the claim by induction over $m.$ For $m=0,$ suppose  $K=x_{0,0}^ax_{0,1}^bx_{0,2}^c.$ Since $K\in M,$ we have $a+b\le p-1$ and $b+c\le p-1.$  In this case we can construct $\max\{a,c\}+b\le p-1$ sets $S_r$ that satisfy the hypothesis of the claim.  Suppose that the claim is true for some $m,$ and $K=K'x_{m,0}^ax_{m,1}^bx_{m,2}^cx_{m+1,0}^{a'}x_{m+1,1}^{b'}x_{m+1,2}^{c'}$ where $K'\in \CC[x^{(\le m-1,\le2)}].$ By the induction hypothesis, there exist $\nu\le p-1$ sets $S_0,\ldots,S_{\nu-1}$  satisfying the hypothesis of the claim for the variables appearing in $K'x_{m,0}^ax_{m,1}^bx_{m,2}^c$ considered with their powers. If we suppose that $a\le c,$ then we can assume that for every set $S_r$ that contains $x_{m,2},$ also $x_{m,0}$ is contained in $S_r$; otherwise we bring $x_{m,0}$ from another set to $S_r$ and this will not change the properties of the sets $S_0,\ldots ,S_{\nu-1}.$ Similarly, if $c \le a,$ we can assume that for every set $S_r$ that contains $x_{m, 0},$ also $x_{m, 2}$ is contained in that set.  Let us take the $a'$ copies of $x_{m+1,0}$ and distribute them over the sets that do not contain $x_{m,0},x_{m,1},$ or $x_{m,2}.$ If the number of these sets is less than $a',$ construct new sets containing only the single element $x_{m+1,0}.$ 	Since $a'+b+\max\{a,c\}\le p-1,$ the total number of the sets $S_r$ is still $p-1$ at most. For the $b'$ copies of $x_{m+1,1},$ let us distribute them over the sets that do not contain $x_{m,2}$ nor $x_{m+1,0}.$ If the total number of these latter sets is smaller than $b',$ then construct new sets containing the single element $x_{m+1,1}.$ Since $a'+b'+c\le p-1,$ the total number of the sets $S_r$ is still $ p-1$ at most. We apply the same argumentation to~$c',$ concluding the proof.
	\end{proof}
	Let us fix $S_r=\{x_{i_1',s_1'},\ldots ,x_{i_{n_r}',s_{n_r}'}\}$ such that $x_{i_1',s_1'}\prec x_{i_2',s_2'}\prec \cdots \prec x_{i_{n_r}',s_{n_r}'}.$ Therefore, the elements of $S_r$ satisfy the following:
	\[(i'_{j+1}-i'_j>1) \text{ or } (i'_{j+1}=i'_j+1\, \land \,s'_j\neq 2\, \land\, s'_{j+1}\neq 0) \text{ or }(i'_{j+1}=i'_j\, \land\, s'_j=0\, \land\, s'_{j+1}=2).\]
	We claim that $\phi(x_{i_1',s_1'} x_{i_2',s_2'}\cdots x_{i_{n_r}',s_{n_r}'})$ contains a monomial summand $\alpha_r$ of the form 
	\[\alpha_r=(\eta_{a_1,b_1}^r\wedge \xi_{c_1,d_1}^r)\wedge (\eta_{a_2,b_2}^r\wedge \xi_{c_2,d_2}^r)\wedge\cdots \wedge (\eta_{a_{n_r},b_{n_r}}^r\wedge \xi_{c_{n_r},d_{n_r}}^r),\]
	where $x_{a_j,b_j}\prec x_{a_{j+1},b_{j+1}} $ and $x_{c_j,d_j}\prec x_{c_{j+1},d_{j+1}}.$ Moreover, $a_j+c_j=i'_j$ and $b_j+d_j=s'_j.$ Let $(a_1,b_1)=(i'_1,s'_1)$ and $(c_1,d_1)=(0,0).$ Suppose that  $(\eta_{a_1,b_1}^r\wedge \xi_{c_1,d_1}^r)\wedge\cdots \wedge (\eta_{a_\ell,b_\ell}^r\wedge \xi_{c_\ell,d_\ell}^r)$ is constructed and let us construct $\eta_{a_{\ell+1},b_{\ell+1}}^r\wedge \xi_{c_{\ell+1},d_{\ell+1}}^r.$
	\begin{enumerate}[{Case 1.}]
		\item[Case 1.] If $i'_{\ell+1}-i'_\ell>1,$ then set $(a_{\ell+1},b_{\ell+1})=(a_\ell+i'_{\ell+1}-i'_\ell-1,s'_{\ell+1}) $ and $(c_{\ell+1},d_{\ell+1})=(c_\ell+1,0).$
		\item[Case 2.] If $i'_{\ell+1}=i'_\ell+1,$ then: \begin{itemize}
			\item If $b_\ell=0$ and $d_\ell=1,$ then set $(a_{\ell+1},b_{\ell+1})=(a_\ell,1)$ and $(c_{\ell+1},d_{\ell+1})=(c_\ell+1,s_{\ell+1}-1).$
			\item Otherwise set $(a_{\ell+1},b_{\ell+1})=(a_\ell+1, s_{\ell+1}-1)$ and $(c_{\ell+1},d_{\ell+1})=(c_\ell,1).$
		\end{itemize}
		\item[Case 3.] If $i'_{\ell+1}=i'_\ell,$ which implies $s'_\ell=b_\ell=d_\ell=0,$ then set $(a_{\ell+1},b_{\ell+1})=(a_\ell,1)$ and $(c_{\ell+1},d_{\ell+1})=(c_\ell,1).$
	\end{enumerate}
We repeat the same construction for every $r.$ Then $\alpha=\alpha_0\wedge\alpha_1\wedge\cdots \wedge\alpha_\nu$ is a monomial summand  of $\phi(K).$ Now suppose that there exists another monomial $K^* $ of $L$ such that $\phi(K^*)$ contains $\alpha$ as a monomial summand. Let $x_{i^*_1,s^*_1}$ be the smallest variable appearing in $K^*.$ Then $x_{i^*_1,s^*_1}\preceq x_{i_1,s_1}.$  The expression of $\phi(x_{i^*_1,s^*_1})$  has a summand of the form $\eta_{u,v}^r\wedge\xi_{\tilde{u},\tilde{v}}^r$ for some $r$ that appears in $\alpha$ where $u+\tilde{u}=i^*_1$ and $v+\tilde{v}=s^*_1.$ This means that $x_{i^*_1,s^*_1}$ appears in the monomial $K,$ and therefore $x_{i^*_1,s^*_1}=x_{i_1,s_1}.$  By going through all the variables of $K^*$ and  repeating the same argument, we conclude  $K=K^*.$ Thus $\phi(L)\neq0$ and  $L\notin I_{p,(m+1,3)}.$
	\end{proof}
\begin{proposition}
	If $m\ge 1,$ $n\ge 3,$ and $p\ge 2,$ the family $\{(x^p)^{(k,\ell)}\}_{0\le k\le mp, 0\le \ell \le np}$ is {\em not} a Gr\"obner basis of the differential ideal generated by $x^p$ in the ring $\CC[x^{(\le m, \le n)}]$ with respect to any $T_{m,n}$-ordering.
\end{proposition}
\begin{proof}
If $\{(x^p)^{(k,\ell)}\}_{0\le k\le mp, 0\le \ell \le np}$ is a Gr\"obner basis of the differential ideal generated by $x^p$ w.r.t. the $T_{m,n}$-ordering $\prec,$ the same statement holds for the \mbox{$T_{m-1,n}$-ordering~$\prec.$} Therefore, we  restrict our proof to the case $m=1.$
Consider the differential polynomials $(x^p)^{(p-1,3)}$ and $(x^p)^{(p-1,0)}.$ We will show that their $S$-polynomial does not have an LCM-representation. By \cite[Theorem~2.9.6]{cox2007ideals}, the $(x^p)^{(k,\ell)}$ then are not a Gr\"obner basis.  Note that $\lm((x^p)^{(p-1,3)})=x_{0,3}x_{1,0}^{p-1}$ and $\lm((x^p)^{(p-1,0)})=x_{0,0}x_{1,0}^{p-1}.$ Their least common multiple is 
$\lcm(\lm((x^p)^{(p-1,3)}),\lm((x^p)^{(p-1,0)}))=x_{0,0}x_{0,3}x_{1,0}^{p-1}.$
We proceed with the proof by contradiction. Suppose that 
\[S((x^p)^{(p-1,3)},(x^p)^{(p-1,0)})=\sum_{a,b}(x^p)^{(a,b)}g_{a,b},\]
\vspace*{-2mm}

\noindent where $\lm((x^p)^{(a,b)}g_{a,b})\prec x_{0,0}x_{0,3}x_{1,0}^{p-1}.$ Since all monomials in $S((x^p)^{(p-1,3)},(x^p)^{(p-1,0)})$ are of degree $p+1$ and homogeneous w.r.t. both derivatives $\partial_s$ and $\partial_t,$ we can write 
\[S((x^p)^{(p-1,3)},(x^p)^{(p-1,0)})=\sum_{p-2\le a\le p-1,0\le b\le 3} c_{a,b}(x^p)^{(a,b)}x_{p-1-a,3-b},\]
where $c_{a,b}$ are constants and $(x^p)^{(a,b)}x_{p-1-a,3-b}\prec   x_{0,1}x_{0,3}x_{1,0}^{p-1}.$ We now list the  polynomials that can show up in the previous equality with their leading monomials: 
\begin{align*}
	(x^p)^{(p-2,0)}x_{1,3}, & \quad\lm(	(x^p)^{(p-2,0)}x_{1,3} )=x_{0,0}^2x_{1,0}^{p-2}x_{1,3},\\
	(x^p)^{(p-2,1)}x_{1,2}, & \quad\lm(	(x^p)^{(p-2,1)}x_{1,2} )=x_{0,0}x_{0,1}x_{1,0}^{p-2}x_{1,2},\\
	(x^p)^{(p-2,2)}x_{1,1}, & \quad\lm(	(x^p)^{(p-2,2)}x_{1,1} )=x_{0,1}^2x_{1,0}^{p-2}x_{1,1},\\
	(x^p)^{(p-2,3)}x_{1,0}, &\quad \lm(	(x^p)^{(p-2,3)}x_{1,0} )=x_{0,1}x_{0,2}x_{1,0}^{p-2}x_{1,1},\\
	(x^p)^{(p-1,0)}x_{0,3}, & \quad\lm(	(x^p)^{(p-1,0)}x_{0,3} )=x_{0,0}x_{1,0}^{p-1}x_{0,3},\\
	(x^p)^{(p-1,1)}x_{0,2}, &\quad \lm(	(x^p)^{(p-1,1)}x_{0,2} )=x_{0,1}x_{1,0}^{p-1}x_{0,2},\\
	(x^p)^{(p-1,2)}x_{0,1}, & \quad\lm(	(x^p)^{(p-1,2)}x_{0,1} )=x_{0,2}x_{1,0}^{p-1}x_{0,1},\\
	(x^p)^{(p-1,3)}x_{0,0}, &\quad  \lm(	(x^p)^{(p-1,3)}x_{0,0} )=x_{0,3}x_{1,0}^{p-1}x_{0,0}.
\end{align*}
Among all these polynomials, only $(x^p)^{(p-2,0)}x_{1,3}$ and $(x^p)^{(p-2,1)}x_{1,2}$ have leading monomials smaller than $ x_{0,0}x_{0,3}x_{1,0}^{p-1}$ under the ordering $\prec.$ Thus,
\begin{equation}\label{eq:Spol}
S((x^p)^{(p-1,3)},(x^p)^{(p-1,0)})=c_{p-2,0}(x^p)^{(p-2,0)}x_{1,3}+c_{p-2,1}(x^p)^{(p-2,1)}x_{1,2}.\end{equation}
Note that $x_{0,2}x_{1,0}^{p-2}x_{1,1}$ is a monomial summand  of the polynomial $(x^p)^{(p-1,3)}.$ Then
 $x_{0,0}x_{0,2}x_{1,0}^{p-2}x_{1,1}$ shows up in $S((x^p)^{(p-1,3)},(x^p)^{(p-1,0)})$ but not in $c_{p-2,0}(x^p)^{(p-2,0)}x_{1,3}+c_{p-2,1}(x^p)^{(p-2,1)}x_{1,2},$ which is in contradiction to Equation~\eqref{eq:Spol}.
\end{proof}
\section{Linking $\dim_{\CC}(R_n/C_{p,n})$ and $\dim_{\CC}(R_{m,n}/C_{p,(m,n)})$ to lattice polytopes}\label{section polynomiality}
We now link the sequences $\dim_{\CC}(R_n/C_{p,n})$ and $\dim_{\CC}(R_{m,n}/C_{p,(m,n)})$---both considered as sequences in~$p$---to lattice polytopes.

\subsection{Polynomiality of $\dim_{\CC}(R_n/C_{p,n})$}
We investigate the following question.
\begin{question}\label{conjBernd} 
	Fix $n\in \NN.$ As $p$ varies, is $\left( \dim_{\CC} (R_n/C_{p,n}) \right)_{p\in \NN}$ a polynomial in~$p$ of degree $n+1$?
\end{question}
Before turning to the proof that the answer to this question is {\em yes}, we present an explicit example.
\begin{example}[$\dim_{\CC} (R_6/C_{p,6})_{p\in \mathbb{N}}$]
	Computations in {\tt Singular}~\cite{Singular} reveal the first $13$ entries of the sequence $\dim_{\CC} (R_6/C_{p,6})_{p\in \mathbb{N}}$ to be
	\begin{align*}\begin{split}
	0,1, 34, 353, 2\,037, 8\,272, 26\,585, 72\,302, 17\,3502, 377\,739, 760\,804, 
	1\,437\,799, 2\,576\,795,
	\end{split}\end{align*} 
 coinciding with the sequence \href{https://www.oeis.org/A244881}{www.oeis.org/A244881}. 
With~{\tt Mathematica}, we compute the interpolating polynomial on the values for $p=1,\ldots,20$ to be
$$\frac{17}{315}p^7+\frac{17}{90} p^6+\frac{53}{180}p^5+\frac{19}{72}p^4+\frac{13
	}{90}p^3+\frac{17}{360}p^2+\frac{1}{140}p,$$ which is indeed of degree $7=6+1.$
\end{example}
Let $\prec$ denote the reverse lexicographical ordering on $R_n=\CC[x_0,\ldots,x_n].$ 
The following lemma, as also presented in \cite{BMS}, determines the initial ideal of $C_{p,n}$ w.r.t.~$\prec.$ We present a proof building on Zobnin's result and~\Cref{prop iso}. 
\begin{lemma}\label{lemma inital Ip}
	The initial ideal of $C_{p,n}$ with respect to $\prec$  is generated by the family $\left\{x_i^{u_i}x_{i+1}^{u_{i+1}} \mid  u_i+u_{i+1}=p,  \, 0\leq i\leq n-1\right\}.$
\end{lemma}
\begin{proof}
	Let us first prove that the leading monomials of our family of generators are $x_i^{u_i}x_{i+1}^{u_{i+1}}.$ Let $0\le k< np$ be of the form $k=mp+(p-a),$ where $1\le a \le p$ and $0\le m\le n-1.$ For $k=np,$ the leading term of $f_k$ is $x_n^p,$ where $f_k$ denotes the coefficient of $t^k$ in the polynomial $f_{p,n}.$  We claim that the leading monomial of the polynomial $f_k$ is $x_m^ax_{m+1}^{p-a}.$ Suppose that $x_0^{i_0 }\cdots x_{n}^{i_{n}} \succ x_m^ax_{m+1}^{p-a}$ for some monomial summand $x_0^{i_0}\cdots x_{n}^{i_{n}}$ in $f_k.$ This implies that $i_0=\cdots=i_{m-1}=0.$ Then  $i_m+\cdots +i_{n}=p$ and $mi_m+\cdots +ni_{n}=mp+p-a=k.$ Since $i_m\leq a,$ from 
	$$(a-i_m)+(m+1)(i_{m}+\cdots+i_n-p)+(i_{m+2}+\cdots+(n-m-1)i_n) = 0$$
	 we conclude that $i_{m}=a,$ $i_{m+1}=p-a,$ and $i_{m+2}=\cdots=i_n=0.$ Therefore, $x_m^ax_{m+1}^{p-a}$ is indeed the leading monomial of $f_k.$ We now consider the truncated differential ring $\CC[x^{(\leq n)}].$ As rings, $\CC[x^{\leq n}]\cong \CC[x_0,\ldots,x_n]=R_n.$ Then the following holds:
$$\init_\prec \langle\{(x^p)^{(k)} \mid 0\le k\le np\}\rangle = \langle\{ \lm \left(\sum_{k=0}^{np} [r_k](x^p)^{(k)}\right) \big|r_k\in \CC[x^{(\infty)}]\}\rangle,$$ where $[r_k]$ denotes the equivalence class of $r_k$ in $\CC[x^{\leq n}].$
 	By Zobnin's result, 
	\linebreak$\lm \left(\sum_{k=0}^{np} (x^p)^{(k)}r_k\right)$ is contained in the ideal generated by the family of elements $\{\lm(x^p)^{(k)}\}_{k\in \NN}.$  Therefore, the initial ideal of $\langle \{ (x^p)^{(k)}\}_{0\le k\le np} \rangle $ is generated by \linebreak $\{\lm ((x^p)^{(k)})\}_{0\le k\le np},$ concluding the proof.	
\end{proof}
\begin{lemma}\label{lemma polytope}
	The convex $(n+1)$-dimensional polytope $$P_n  \coloneqq  \left\{(u_0,\ldots,u_n)\in (\mathbb{R}_{\geq 0})^{n+1} \, \big| \, u_i+u_{i+1}\leq 1 \text{ for all } 0\leq i\leq n-1\right\}$$ is a lattice polytope whose vertices are binary vectors with no consecutive $1$s.
\end{lemma}
Before proving the lemma, we recall some definitions. Let $G=(V,E)$ be an  undirected graph, where $V$ denotes the set of vertices and $E$ the set of edges. 
A graph is {\em perfect} if for every subgraph, the chromatic number equals the clique number of that subgraph.
A subset $S\subseteq V$ of  vertices is called {\em stable} if no two elements of $S$ are adjacent. 
 Borrowing the notation from~\cite{gro99}, the {\em stable set polytope of $G$} is the polytope
$$ \STAB(G)\coloneqq \conv \left\{ \chi^S \in \RR^V \mid S \subseteq V \text{ stable}\right\},$$
where the {\em incidence vectors}  $\chi^S=(\chi^S_v)_{v\in V}\in \RR^V$ are defined as  
$$ \chi_v^S \coloneqq \begin{cases}
1 & \text{if } v\in S,\\
0 & \text{else}.
\end{cases} $$
The {\em fractional stable set polytope of $G$} is defined as
$$ \QSTAB(G)\coloneqq \left\{ x \in \RR^V \mid 0 \leq x(v)\, \forall v \in V, \ \sum_{v\in Q} x(v) \leq 1 \text{ for all cliques } Q \text{ of }G \right\}.$$
Hence $\STAB(G) = \conv\{ \{0,1\}^V\cap \QSTAB(G) \}.$
Chv\'{a}tal \cite[Theorem 3.1]{ch75} proved that a graph $G$ is perfect if and only if $\STAB(G)=\QSTAB(G).$ If follows from Fulkerson's theory of anti-blocking polyhedra~\cite{Fulkerson} that this result is equivalent to the perfect graph theorem. The latter was conjectured by Berge~\cite{berge} and proven by Lov\'asz~\cite{Lo72}. 
 \begin{proof}[Proof of \Cref{lemma polytope}]
Consider the graph $G=(\{0,1, \ldots, n\}, \{ [i,i+1]\}_{i=0,\ldots,n-1}).$ 
Observe that $P_n$ is precisely  the fractional stable set polytope of $G.$ Since $G$ is a perfect graph, $\QSTAB(G)=\STAB(G)$ and $P_n$ has zero-one vertices as claimed.
\end{proof}
For a $d$-dimensional polytope $P\subseteq \RR^n$ with integer vertices and $t\in \NN,$ denote by $L_P(t)\coloneqq \vert tP\cap \ZZ^n \! \vert$ the number of lattice points of the dilated polytope~$tP.$ Ehrhart proved that this number is a rational polynomial in~$t$ of degree~$d,$ i.e., there exist rational numbers $l_{P,0},\ldots, l_{P,d},$ s.t. 
$$ L_P(t)  =  l_{P,d}t^d +\cdots + l_{P,1}t + l_{P,0} . $$
The polynomial $L_P \in \QQ[t]  $ is called the {\em Ehrhart polynomial} of~$P.$ 
\begin{theorem}\label{thm Ehrhart}
	The number	$\dim_{\CC}(R_n/C_{p,n})$ is the Ehrhart polynomial of the polytope~$P_n$ defined in \Cref{lemma polytope} evaluated at $p-1.$
\end{theorem}
\begin{proof}
From \Cref{lemma inital Ip} we read that $x_0^{u_0}\cdots x_n^{u_n}$ is a standard monomial if and only if $u_i+u_{i+1}<p$ for all $0\leq i \leq n-1.$ The claim then follows from \Cref{lemma polytope}.
\end{proof}

\subsection{Investigation of $\dim_{\CC}(R_{m,2}/C_{p,(m,2)})$}\label{section two var}
In this section, we generalize the results found for $R_n/C_{p,n}$  to two independent variables, i.e., to $R_{m,n}/C_{p,(m,n)}.$ 
	\begin{theorem}\label{prop dimmultivar} For all $m,p\in \NN,$ $\dim_{\CC} ( R_{m,2}/C_{p,(m,2)})$ is the Ehrhart polynomial of the \mbox{$3(m+1)$-dimensional} lattice polytope
	\begin{align*}
P_{(m,2)}  \coloneqq  &\big\{(u_{00},u_{01},u_{02},\ldots,u_{m0},u_{m1},u_{m2})\in (\mathbb{R}_{\geq 0})^{3(m+1)} \big| u_{k_1,l_1}+u_{k_2,l_2}+u_{k_3,l_3}\leq 1\\ &\ \text{for all indices s.t. }\{(k_1,l_1),(k_2,l_2),(k_3,l_3)\} \text{ is a triangle of } T_{m,2} \big\}
\end{align*}
	evaluated at~$p-1.$
\end{theorem}
\begin{proof} 
	Let $G$ be the edge graph of the regular triangulation from \Cref{fig reg triang m2} for \mbox{$m=2$} with $3(m+1)$ vertices and $2+7m$ edges. Since this graph is perfect and the maximal cliques are precisely the triangles of $T_{m,2},$ $\STAB(G)=\QSTAB(G)=P_{(m,2)}.$ By \Cref{thm zobnin m2}, $x_{00}^{u_{00}}\cdots x_{m2}^{u_{m2}}$ is a standard monomial if and only if for all triples of indices $\{(i_1,j_1),(i_2,j_2),(i_3,j_3)\}$ that are a triangle of $T_{m,2},$  $u_{i_1,j_1}+u_{i_2,j_2}+u_{i_3,j_3}\leq p-1.$
\end{proof}
In terms of integer programming, \Cref{prop dimmultivar} translates as follows.
\begin{corollary}
For all $m\in \NN,$ $\dim_{\CC}(R_{m,2}/C_{p,(m,2)})$  is polynomial in $p$ of degree $3(m+1).$ For fixed $p$, this is the number of non-negative integer solutions of the $2m^2$ linear inequalities $x_{i_1,j_1} + x_{i_2,j_2} + x_{i_3,j_3} \leq  p-1,$
	where $\{(i_1,j_1),(i_2,j_2),(i_3,j_3)\}$ runs over the $2m^2$ many triangles of $T_{m,2}.$
\end{corollary}

\section{Regular unimodular triangulations of the $m\times n$ rectangle}\label{section regtrg}
We now outline how regular unimodular triangulations of the $m \times n$ rectangle give rise to $T$-orderings on the truncated partial differential ring $\CC[x^{(\leq m,\leq n)}],$ or, equivalently, on the polynomial ring $\CC[\{x_{k,\ell}\}_{0\leq k\leq m,0\leq \ell \leq n}].$
\begin{example}[$m=n=2$] Again, denote by $C_{p,(2,2)}$ the ideal in 
	$$R_{2,2} = \CC[x_{00}, x_{01},x_{02},x_{10},x_{11},x_{12},x_{20},x_{21},x_{22}]$$ generated by the $(2p+1)^2$ many coefficients $f_{k,\ell}$ of $s^kt^\ell$ in 
	$$ f_{p,(2,2)}  = (   x_{00} + x_{01}t +x_{02}t^2 +x_{10}s+ x_{11}st + x_{12}st^2 +x_{20}s^2 +x_{21}s^2t+x_{22}s^2t^2 )^p.$$
	Let $\prec$ denote the weighted reverse lexicographical ordering on $R_{2,2}$ for the weight
	$$w_{2,2}\coloneqq \left( 2^8+1,\ldots,2^8+1\right)- \left(1,2,2^2,\ldots ,2^8\right)  \in \NN^9,$$
	i.e., assigning weight $128$ to $x_{00},$ weight $127$ to $x_{01},$ and so on. Note that $w_{2,2}$ induces the triangulation $T_{2,2}$ in the {\em upper} hull convention.
For $p=3,$ we find that within the monomials of the $f_{k,\ell},$ the following $8$ triples of pairwise different variables show up:
	\begin{align*}
	&\{ x_{00}, x_{01}, x_{10}\}, 
	\{x_{01}, x_{02}, x_{10}\},
	\{x_{02}, x_{10}, x_{11}\},
	\{x_{02}, x_{11}, x_{12}\},\\
	&\{x_{10}, x_{11}, x_{20}\},
	\{x_{11}, x_{12}, x_{20}\},
	\{x_{12}, x_{20}, x_{21}\},
	\{x_{12}, x_{21}, x_{22}\},
	\end{align*}
the indices of each of which define a triangle of~$T_{2,2}.$
Computations in {\tt Singular} prove that the coefficients of $f_{3,(2,2)}$ are a Gr\"{o}bner basis of $C_{3,(2,2)} \triangleleft R_{2,2}$ with respect to the weighted reverse lexicographical ordering for~$w_{2,2}.$ Our computations for $p\leq 25$ approve the same statements, confirming \Cref{prop Tm2ord} and \Cref{thm zobnin m2}.
There are $64$ regular unimodular triangulations of the $2\times 2$ square in total, four of which give rise to a Gr\"obner basis in the sense above. Those are depicted in \Cref{fig all22}. 
\begin{figure}
$\begin{tikzpicture}[scale=1.2]
\filldraw (0,0) circle (2pt);
\filldraw (1,0) circle (2pt);
\filldraw (2,0) circle (2pt);
\filldraw (0,1) circle (2pt);
\filldraw (0,2) circle (2pt);
\filldraw (1,1) circle (2pt);
\filldraw (1,2) circle (2pt);
\filldraw (2,1) circle (2pt);
\filldraw (2,2) circle (2pt);
\draw (0,0)--(2,0);
\draw (0,0)--(0,2);
\draw (2,0)--(2,2);
\draw  (0,2)--(2,2);
\draw  (1,0)--(0,1);
\draw  (0,2)--(0,1);
\draw (0,2)--(2,0);
\draw (1,0)--(1,2);
\draw  (1,0)--(0,2);
\draw  (2,1)--(1,2);
\draw (1,2)--(2,0);
\end{tikzpicture}$ \quad 
$\begin{tikzpicture}[scale=1.2]
\filldraw (0,0) circle (2pt);
\filldraw (1,0) circle (2pt);
\filldraw (2,0) circle (2pt);
\filldraw (0,1) circle (2pt);
\filldraw (0,2) circle (2pt);
\filldraw (1,1) circle (2pt);
\filldraw (1,2) circle (2pt);
\filldraw (2,1) circle (2pt);
\filldraw (2,2) circle (2pt);
\draw (0,0)--(2,0);
\draw (0,0)--(0,2);
\draw (2,0)--(2,2);
\draw (0,2)--(2,2);
\draw (0,1)--(1,2);
\draw (0,1)--(2,2);
\draw (0,0)--(2,2);
\draw (0,0)--(2,1);
\draw (1,0)--(2,1);
\draw (0,1)--(2,1);
\end{tikzpicture}$
\quad 
$\begin{tikzpicture}[scale=1.2]
\filldraw (0,0) circle (2pt);
\filldraw (1,0) circle (2pt);
\filldraw (2,0) circle (2pt);
\filldraw (0,1) circle (2pt);
\filldraw (0,2) circle (2pt);
\filldraw (1,1) circle (2pt);
\filldraw (1,2) circle (2pt);
\filldraw (2,1) circle (2pt);
\filldraw (2,2) circle (2pt);
\draw (0,0)--(2,0);
\draw (0,0)--(0,2);
\draw (2,0)--(2,2);
\draw (0,2)--(2,2);
\draw (0,1)--(1,2);
\draw (0,0)--(1,2);
\draw (1,0)--(1,2);
\draw (1,0)--(2,1);
\draw (1,0)--(2,2);
\draw (0,0)--(2,2);
\end{tikzpicture}$
\quad 
$\begin{tikzpicture}[scale=1.2]
\filldraw (0,0) circle (2pt);
\filldraw (1,0) circle (2pt);
\filldraw (2,0) circle (2pt);
\filldraw (0,1) circle (2pt);
\filldraw (0,2) circle (2pt);
\filldraw (1,1) circle (2pt);
\filldraw (1,2) circle (2pt);
\filldraw (2,1) circle (2pt);
\filldraw (2,2) circle (2pt);
\draw (0,0)--(2,0);
\draw (0,0)--(0,2);
\draw (2,0)--(2,2);
\draw (0,2)--(2,2);
\draw (0,1)--(2,1);
\draw (0,1)--(2,0);
\draw (1,0)--(0,1);
\draw (0,2)--(2,0);
\draw (0,2)--(2,1);
\draw (0,2)--(2,2);
\draw (1,2)--(2,1);
\end{tikzpicture}$
\caption{\small The four regular unimodular triangulations of the $2\times 2$ square giving rise to a Gr\"obner basis, the first of which is $T_{2,2}.$ Note that they all arise from $T_{2,2}$ by rotating and flipping.\normalsize}
\label{fig all22}
\end{figure}
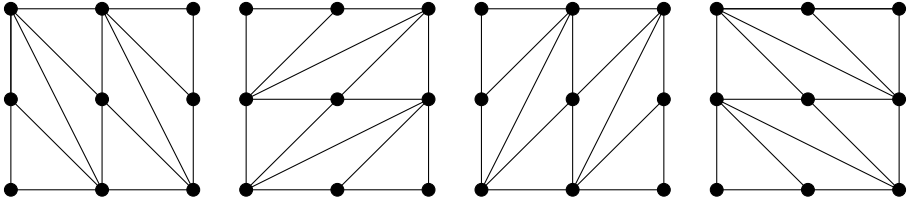
	\end{example}
\begin{example}[$m=3,n=2$] Again, denote by $C_{p,(3,2)}$ the ideal in 
	$$R_{3,2} = \CC\left[x_{00}, x_{01},x_{02},x_{10},x_{11},x_{12},x_{20},x_{21},x_{22},x_{30},x_{31},x_{32}\right]$$ generated by the $(3p+1)(2p+1)$ coefficients of $ f_{p,(3,2)}.$
	Let $\prec$ denote the weighted reverse lexicographical ordering on $R_{3,2}$ for the weight vector
	$$w_{3,2} \coloneqq  (2^{11}+1, \ldots, 2^{11}+1) - (2^0,2^1,\ldots,2^{11})\in \NN^{12}.$$
	For $p=3,$ the following $12$ triples of pairwise different variables show up within the leading monomials of the $70$ coefficients:
	\begin{align*}
	&\{ x_{00}, x_{01}, x_{10}\}, 
	\{x_{01}, x_{02}, x_{10}\},
	\{x_{02}, x_{10}, x_{11}\},
	\{x_{02}, x_{11}, x_{12}\},\\
	&\{x_{10}, x_{11}, x_{20}\},
	\{x_{11}, x_{12}, x_{20}\},
	\{x_{12}, x_{20}, x_{21}\},
	\{x_{12}, x_{21}, x_{22}\},\\	
	&\{x_{20}, x_{21}, x_{30}\},
	\{x_{21}, x_{22}, x_{30}\},
	\{x_{22}, x_{30}, x_{31}\},
	\{x_{22}, x_{31}, x_{32}\},
	\end{align*}
	the indices of each of which define a triangle of~$T_{3,2}.$ 
	Computations in {\tt Singular} prove that the $70$ coefficients $f_{k,\ell}$ of $s^kt^\ell$ in $f_{3,(3,2)}$ are indeed a Gr\"{o}bner basis of $C_{3,(3,2)} \triangleleft R_{3,2}$ w.r.t.~$\prec,$ turning $\prec$ into a $T_{3,2}$-ordering for $p=3.$
	Note that there are $852$ regular unimodular triangulations of the $3\times 2$ rectangle in total,  four of which give rise to a Gr\"obner basis in the sense above. Those four are depicted in \Cref{fig all32}. 
	\begin{figure}[h]
$\begin{tikzpicture}
\filldraw (0,0) circle (2pt);
\filldraw (1,0) circle (2pt);
\filldraw (2,0) circle (2pt);
\filldraw (0,1) circle (2pt);
\filldraw (0,2) circle (2pt);
\filldraw (1,1) circle (2pt);
\filldraw (1,2) circle (2pt);
\filldraw (2,1) circle (2pt);
\filldraw (2,2) circle (2pt);
\filldraw (3,0) circle (2pt);
\filldraw (3,1) circle (2pt);
\filldraw (3,2) circle (2pt);
\draw (0,0)--(3,0);
\draw (0,0)--(0,2);
\draw (3,0)--(3,2);
\draw (0,2)--(3,2);
\draw (1,0)--(0,1);
\draw (0,2)--(0,1);
\draw (0,2)--(2,0);
\draw (0,2)--(1,0);
\draw (1,0)--(1,2);
\draw (2,0)--(2,2);
\draw (1,2)--(3,0);
\draw (1,2)--(2,0);
\draw (2,2)--(3,1);
\draw (2,2)--(3,0);
\end{tikzpicture}$ \ 
$\begin{tikzpicture}
\filldraw (0,0) circle (2pt);
\filldraw (1,0) circle (2pt);
\filldraw (2,0) circle (2pt);
\filldraw (0,1) circle (2pt);
\filldraw (0,2) circle (2pt);
\filldraw (1,1) circle (2pt);
\filldraw (1,2) circle (2pt);
\filldraw (2,1) circle (2pt);
\filldraw (2,2) circle (2pt);
\filldraw (3,0) circle (2pt);
\filldraw (3,1) circle (2pt);
\filldraw (3,2) circle (2pt);
\draw (0,0)--(3,0);
\draw (0,0)--(0,2);
\draw (3,0)--(3,2);
\draw (0,2)--(3,2);
\draw (0,1)--(1,2);
\draw (0,0)--(1,2);
\draw (0,0)--(2,2);
\draw (1,0)--(3,2);
\draw (1,0)--(1,2);
\draw (2,0)--(2,2);
\draw (3,0)--(3,2);
\draw (1,0)--(2,2);
\draw (2,0)--(3,2);
\draw (2,0)--(3,1);
\end{tikzpicture}$ \
$\begin{tikzpicture}
\filldraw (0,0) circle (2pt);
\filldraw (1,0) circle (2pt);
\filldraw (2,0) circle (2pt);
\filldraw (0,1) circle (2pt);
\filldraw (0,2) circle (2pt);
\filldraw (1,1) circle (2pt);
\filldraw (1,2) circle (2pt);
\filldraw (2,1) circle (2pt);
\filldraw (2,2) circle (2pt);
\filldraw (3,0) circle (2pt);
\filldraw (3,1) circle (2pt);
\filldraw (3,2) circle (2pt);
\draw (0,0)--(3,0);
\draw (0,0)--(0,2);
\draw (3,0)--(3,2);
\draw (0,2)--(3,2);
\draw (0,2)--(2,0);
\draw (1,2)--(3,0);
\draw (0,2)--(3,0);
\draw (1,0)--(0,1);
\draw (0,1)--(1,1);
\draw (0,1)--(2,0);
\draw (1,2)--(3,1);
\draw (2,2)--(3,1);
\draw (2,1)--(3,1);
\draw (1,1)--(3,0);
\draw (0,2)--(2,1);
\end{tikzpicture}$\ \ 
$\begin{tikzpicture}
\filldraw (0,0) circle (2pt);
\filldraw (1,0) circle (2pt);
\filldraw (2,0) circle (2pt);
\filldraw (0,1) circle (2pt);
\filldraw (0,2) circle (2pt);
\filldraw (1,1) circle (2pt);
\filldraw (1,2) circle (2pt);
\filldraw (2,1) circle (2pt);
\filldraw (2,2) circle (2pt);
\filldraw (3,0) circle (2pt);
\filldraw (3,1) circle (2pt);
\filldraw (3,2) circle (2pt);
\draw (0,0)--(3,2);
\draw (0,0)--(2,2);
\draw (0,0)--(2,1);
\draw (0,1)--(1,1);
\draw (1,1)--(3,2);
\draw (1,0)--(3,2);
\draw (2,1)--(3,1);
\draw (2,0)--(3,1);
\draw (0,0)--(0,2);
\draw (0,2)--(3,2);
\draw (0,0)--(3,0);
\draw (3,0)--(3,2);
\draw (0,1)--(1,2);
\draw (0,1)--(2,2);
\draw (1,0)--(3,1);
\end{tikzpicture}$ 
		\caption{\small The four regular unimodular triangular regulations of the $3\times 2$ rectangle giving rise to a Gr\"obner basis, the first of which is $T_{3,2}.$\normalsize}
		\label{fig all32}
	\end{figure}
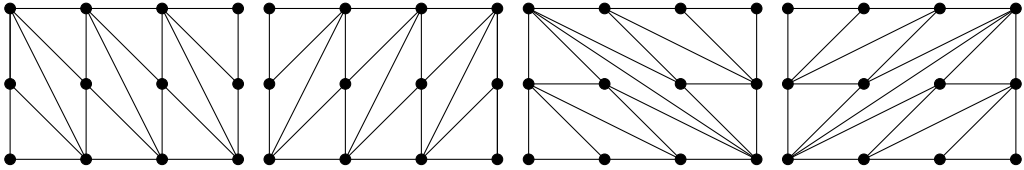
\end{example}
\begin{remark} Computations finish only for small values of $m$ and $p$.
	For $m=5$, we validated with {\tt Singular} that the coefficients of $f_{p,(5,2)}$ are a Gröbner basis with respect to the weighted reverse lexicographical ordering $\prec$ for a vector inducing $T_{5,2}$ in the upper hull convention up to $p= 9,$ approving that $\prec$ is a $T_{5,2}$-ordering for $p\leq 9.$ For $m=8$, we validated this for $p\leq 6.$ For greater values, even though computing over finite characteristic, the computations did not terminate within several days.
\end{remark}
\begin{question}\label{question 1}
For which $m,n,p\in  \NN$ does there exist a regular unimodular triangulation $T$ of the $m\times n$ rectangle such that the coefficients of $f_{p,(m,n)}$ are a Gr\"{o}bner basis of $C_{p,(m,n)}$  with respect to the weighted reverse lexicographical ordering for a vector inducing that triangulation in the {\em upper} hull convention?
\end{question} 
As before, let $T_{m,n}$ be the placing triangulation of the point configuration\linebreak
$\left[(0,0),(0,1),\ldots,(0,n),\ldots,(m,0),\ldots, (m,n)\right].$
We point out that this triangulation does {\em not} lead to a positive answer of Question~\ref{question 1} in general. For instance, $T_{1,3}$ does not give rise to a Gr\"obner basis; only the four triangulations depicted in \Cref{fig all13}~do.

\begin{figure}[h]
		\floatbox[{\capbeside\thisfloatsetup{capbesideposition={right, center},capbesidewidth=8cm}}]{figure}[\FBwidth]
	{\caption{\small The four regular unimodular triangular regulations of the $1\times 3$ rectangle  that give rise to a Gr\"obner basis.\normalsize}\label{fig all13}}{
	$\hspace*{1cm}\begin{tikzpicture}[scale=1.0]
	\filldraw (0,0) circle (2pt);
	\filldraw (1,0) circle (2pt);
	\filldraw (0,1) circle (2pt);
	\filldraw (0,2) circle (2pt);
	\filldraw (0,3) circle (2pt);
	\filldraw (1,1) circle (2pt);
	\filldraw (1,2) circle (2pt);
	\filldraw (1,3) circle (2pt);
	\draw (0,0)--(0,3);
	\draw (0,0)--(1,0);
	\draw (1,0)--(1,3);
	\draw  (0,3)--(1,3);
	\draw  (0,3)--(1,2);
	\draw  (0,2)--(1,1);
	\draw  (0,1)--(1,0);
	\draw  (0,1)--(1,1);
	\draw (0,2)--(1,2);
	\end{tikzpicture}$ \ 
	$\begin{tikzpicture}[scale=1.0]
	\filldraw (0,0) circle (2pt);
	\filldraw (1,0) circle (2pt);
	\filldraw (0,1) circle (2pt);
	\filldraw (0,2) circle (2pt);
	\filldraw (0,3) circle (2pt);
	\filldraw (1,1) circle (2pt);
	\filldraw (1,2) circle (2pt);
	\filldraw (1,3) circle (2pt);
	\draw (0,0)--(0,3);
	\draw (0,0)--(1,0);
	\draw (1,0)--(1,3);
	\draw  (0,3)--(1,3);
	\draw (0,0)--(1,1);
	\draw (0,1)--(1,2);
	\draw (0,2)--(1,3);
	\draw  (0,1)--(1,1);
	\draw (0,2)--(1,2);
	\end{tikzpicture}$ \ 
	$\begin{tikzpicture}[scale=1.0]
	\filldraw (0,0) circle (2pt);
	\filldraw (1,0) circle (2pt);
	\filldraw (0,1) circle (2pt);
	\filldraw (0,2) circle (2pt);
	\filldraw (0,3) circle (2pt);
	\filldraw (1,1) circle (2pt);
	\filldraw (1,2) circle (2pt);
	\filldraw (1,3) circle (2pt);
	\draw (0,0)--(0,3);
	\draw (0,0)--(1,0);
	\draw (1,0)--(1,3);
	\draw  (0,3)--(1,3);
	\draw (0,1)--(1,0);
	\draw  (0,2)--(1,0);
	\draw  (0,2)--(1,1);
	\draw  (0,3)--(1,2);
	\draw  (0,3)--(1,1);
	\end{tikzpicture}$ \ 
	$\begin{tikzpicture}[scale=1.0]
	\filldraw (0,0) circle (2pt);
	\filldraw (1,0) circle (2pt);
	\filldraw (0,1) circle (2pt);
	\filldraw (0,2) circle (2pt);
	\filldraw (0,3) circle (2pt);
	\filldraw (1,1) circle (2pt);
	\filldraw (1,2) circle (2pt);
	\filldraw (1,3) circle (2pt);
	\draw (0,0)--(0,3);
	\draw (0,0)--(1,0);
	\draw (1,0)--(1,3);
	\draw  (0,3)--(1,3);
	\draw  (0,0)--(1,1);
	\draw (0,0)--(1,2);
	\draw  (0,1)--(1,2);
	\draw (0,1)--(1,3);
	\draw  (0,2)--(1,3);
	\end{tikzpicture}$\hspace*{-1cm}}
\end{figure}
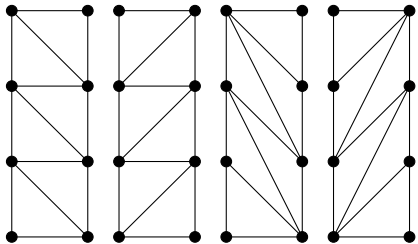

For $m=n=3,$ the question has a negative answer.
There are in total $46\,452$ regular unimodular triangulations of the $3\times 3$ square. The coefficients of $f_{3,(3,3)}$  are not a Gr\"{o}bner basis of $C_{3,(3,3)}$ w.r.t. the weighted reverse lexicographical ordering for any of the vectors in the strict interior of the secondary cone of those triangulations.  
\begin{remark}
As pointed out in~\cite{BJMS15}, there are---up to symmetries---$5\,941$ regular unimodular triangulations of the $3\times 3$ square. It would actually be sufficient to check the Gr\"obner basis property for each of those.
\end{remark} It would be intriguing to find the reason for this failure and to determine all $m,n\in \NN$ for which Question~\ref{question 1} has a positive answer. Let us point out that this problem gets computationally expensive quickly: for the $4\times 2$ rectangle, there are $12\,170$ regular unimodular triangulations, whereas for the $4\times 3$ rectangle, there are already $2\,822\,146.$

Now let $T$ be a triangulation of the $m\times n$ rectangle as asked for in Question~\ref{question 1}. We end this article with two questions.
\begin{question}
Which unimodular regular triangulations of the $m\times 2 $ rectangle give rise to a Gr\"obner basis---can they be deduced from the triangulations found for $m=2,3$  (cf. Figures \ref{fig all22} and \ref{fig all32})?
\end{question}
\begin{question}
 As $p$ varies, is $\dim_{\CC}(R_{m,n}/C_{p,(m,n)})$ the Ehrhart polynomial of the fractional stable set polytope of the edge graph of $T$ and is this graph perfect?
\end{question} 
\smallskip
\small
\bibliography{literature}
\bibliographystyle{abbrv} 
\bigskip
\normalsize
\end{document}